\documentclass[11pt]{article}%
\usepackage{amssymb,amsmath,amsfonts,amsthm,array,bm,bbm,color}%
\usepackage{graphicx}
\providecommand{\U}[1]{\protect\rule{.1in}{.1in}}

\numberwithin{equation}{section}

\newtheorem{theorem}{Theorem}[section]
\newtheorem{lemma}[theorem]{Lemma}
\newtheorem{corollary}[theorem]{Corollary}
\newtheorem{proposition}[theorem]{Proposition}
\newtheorem{remark}[theorem]{Remark}

\newtheorem{definition}[theorem]{Definition}

\def\<{\langle}
\def\>{\rangle}
\def\d{{\rm d}}

\def\E{\mathbb{E}}
\def\N{\mathbb{N}}
\def\P{\mathbb{P}}
\def\R{\mathbb{R}}
\def\T{\mathbb{T}}
\def\Z{\mathbb{Z}}
\def\F{\mathcal{F}}

\def\eps{\varepsilon}

\begin{document}

\title{Convergence of stochastic 2D inviscid Boussinesq equations with transport noise to a deterministic viscous system
}
\author{Dejun Luo\footnote{Email: luodj@amss.ac.cn. Key Laboratory of RCSDS, Academy of Mathematics and Systems Science, Chinese Academy of Sciences, Beijing, China and School of Mathematical Sciences, University of the Chinese Academy of Sciences, Beijing, China.} }

\maketitle

\vspace{-20pt}

\begin{abstract}
The inviscid 2D Boussinesq system with thermal diffusivity and multiplicative noise of transport type is studied in the $L^2$-setting. It is shown that, under a suitable scaling of the noise, weak solutions to the stochastic 2D Boussinesq equations converge weakly to the unique solution of the deterministic viscous Boussinesq system. Consequently, the transport noise asymptotically regularizes the inviscid 2D Boussinesq system and enhances dissipation in the limit.
\end{abstract}



\section{Introduction}

Recently, some scaling limits have been proved for the vorticity form of stochastic two dimensional (2D) Euler equations on the torus perturbed by multiplicative noise of transport type, both in the white noise regime and in the more regular $L^2$-regime. On the one hand, it was shown in \cite{FlaLuoAoP} that, under a suitable scaling of the noise, the white noise solutions of stochastic 2D Euler equations converge weakly to the unique stationary solution of the 2D Navier-Stokes equation driven by space-time white noise; the latter has been studied intensively in the last three decades, cf. \cite{AC90, AF, DaPD}. On the other hand, in the more regular $L^2$-regime, we proved in \cite{FGL} that the limit equation is the vorticity form of the deterministic 2D Navier-Stokes equations; see \cite{Galeati} for an earlier work on linear transport equations and \cite{LuoSaal, LuoZhu} for similar scaling limits in both regimes on the stochastic modified Surface Quasi-Geostrophic equations. A common feature in the above limit results is that the uniqueness of solutions to approximating nonlinear equations is unknown, while the limit equations are uniquely solvable in suitable sense. This shows the approximative regularizing effect of transport noises. Moreover, transport noise, which is formally energy-preserving, tends to dissipate energy in the limit process and thus it exhibits mixing property of the fluid. Another point, particularly relevant in the $L^2$-regime, is that larger noise intensity leads to greater viscosity coefficient in the deterministic limit equations. This is closely related to the phenomenon of dissipation enhancement, which has been studied extensively in the deterministic setting, see e.g. \cite{Constantin, Iyer} and the references therein. In the recent paper \cite{FlaLuo19}, we have explored this idea to show that transport noise provides a bound on the vorticity for stochastic 3D Navier-Stokes equations with unitary viscosity, yielding long-term existence of solutions for large initial data, with large probability. The above limit results remind us of the analog with the theory of stabilization by noise \cite{Arnold, ArnoldCW}, which shows that suitable noises stabilize some finite dimensional linear ODE with a coefficient matrix of negative trace. The purpose of the present work is to establish a similar scaling limit for the 2D Boussinesq system.

\subsection{Deterministic 2D Boussinesq system}

The 2D Boussinesq system describes the evolution of the velocity field $u$ of an incompressible fluid under a vertical force which is proportional to some scalar field $\theta$ (e.g., the temperature), the latter being transported by $u$. The viscous Boussinesq system with thermal diffusivity reads as
  \begin{equation}\label{2D-Boussinesq}
  \left\{\aligned
  & \partial_t \theta + u\cdot\nabla \theta = \kappa\Delta \theta,\\
  & \partial_t u + u\cdot\nabla u + \nabla p= \delta \Delta u + \theta e_2, \\
  & \nabla \cdot u=0,
  \endaligned \right.
  \end{equation}
where $\kappa, \delta\geq 0$ are the diffusion coefficient and the kinematic viscosity, respectively; $p$ is the pressure field, $e_2=(0,1)$ and thus $\theta e_2$ represents the buoyancy force. The interested readers can refer to \cite{Majda} for geophysical background of the Boussinesq system.

When both $\kappa$ and $\delta$ in \eqref{2D-Boussinesq} are positive constants, the existence and uniqueness of finite energy solutions are well known, see e.g. \cite{Cannon-DiBene, Guo}. According to Moffatt's paper \cite{Moffatt}, it is important to study the regularity problem of Boussinesq system in the cases of zero diffusivity ($\kappa=0$) and zero viscosity ($\delta=0$). On one hand, if $\kappa =0$ and $\delta >0$, it was proved in \cite{Danchin-Paicu08} that the viscous system also admits a unique global solution for any initial data with finite energy; therefore, in this case, the system \eqref{2D-Boussinesq} enjoys similar well posedness results as the 2D Navier-Stokes equations, though the proofs in \cite{Danchin-Paicu08} make use of some new and more technical tools. On the other hand, the theory on the zero viscosity case is less complete. If $\kappa>0$ and $\delta=0$, and the initial data $(\theta_0, u_0)$ have $H^m$-regularity for some integer $m>2$, the global well posedness of \eqref{2D-Boussinesq} was proved in \cite[Theorem 1.2]{Chae}. This result was extended in \cite{Hmi-Ker}  to the case where the initial velocity $u_0\in B^{1+2/p}_{p,1}$ and the initial temperature $\theta_0\in L^r$ for some $2<r\leq p<\infty$.

In the current work we are concerned with the zero viscosity case, namely, the constants $\kappa>0$ and $\delta=0$ in \eqref{2D-Boussinesq}. We restrict ourselves to the torus $\T^2= \R^2/\Z^2$ with periodic boundary condition, and reformulate the velocity equation in vorticity form:
  \begin{equation}\label{2D-Boussinesq-1}
  \left\{\aligned
  & \partial_t \theta + u\cdot\nabla \theta = \kappa\Delta \theta,\\
  & \partial_t\omega + u\cdot\nabla \omega = \partial_1\theta, \\
  & u= K\ast \omega,
  \endaligned \right.
  \end{equation}
where $\partial_i= \partial_{x_i}$, $\omega= \nabla^\perp \cdot u= \partial_1 u_2 - \partial_2 u_1$ is the vorticity, $K$ is the Biot-Savart kernel on $\T^2$ and $\ast$ the convolution operation. Note that the system preserves the average on $\T^2$ of solutions $\theta$ and $\omega$, thus we shall always assume that $\theta_0$ and $\omega_0$ have zero average on $\T^2$. Moreover, if $\theta \equiv 0$ then the system \eqref{2D-Boussinesq-1} reduces to the vorticity form of 2D Euler equations. We will work in the $L^2$-regime, namely, the initial data $(\theta_0,\omega_0)$ are square integrable functions on $\T^2$. Similarly to the theory of 2D Euler equations, it is easy to show that \eqref{2D-Boussinesq-1} admits a weak solution for any $L^2$-initial data $(\theta_0,\omega_0)$, but the uniqueness is open. In \cite{Danchin-Paicu}, Danchin and Paicu proved global well posedness of \eqref{2D-Boussinesq-1} for Yudovich-type initial data, i.e., $\omega_0 \in L^\infty(\T^2)$ and the initial temperature $\theta_0$ fulfils a natural additional condition; see \cite{Paicu-Zhu} for related recent results.

\subsection{Our model and main result}

Let us first give a very rough derivation of the stochastic model studied in the paper, which comes from the separation of scales and approximating small scales by noise, cf. \cite[Section 1.2]{FlaLuo19} or \cite{MajdaTV} for similar discussions. We decompose the initial condition as follows:
  $$\theta_0 = \theta_{R,0} + \theta_{L,0}, \quad \omega_0 = \omega_{R,0} + \omega_{L,0}, $$
where the subscript $R$ stands for the ``Resolved'' scale part and $S$ the ``Small'' scale part. Assume that one can solve the system
  $$  \left\{\aligned
  & \partial_t \theta_R + u\cdot\nabla \theta_R = \kappa\Delta \theta_R, \quad \theta_R|_{t=0} = \theta_{R,0},\\
  & \partial_t\omega_R + u\cdot\nabla \omega_R = \partial_1\theta_R, \quad \omega_R|_{t=0} = \omega_{R,0},\\
  & \partial_t \theta_S + u\cdot\nabla \theta_S = \kappa\Delta \theta_S, \quad \theta_S|_{t=0} = \theta_{S,0} ,\\
  & \partial_t\omega_S + u\cdot\nabla \omega_S = \partial_1\theta_S, \quad \omega_S|_{t=0} = \omega_{S,0}
  \endaligned \right. $$
with $u= K\ast (\omega_R + \omega_S)$; then the sums $\theta = \theta_R + \theta_S$ and $\omega= \omega_R + \omega_S$ solve \eqref{2D-Boussinesq-1}. Compared to the large scale component $(\theta_R, \omega_R)$, in a certain limit, the small scale part $(\theta_S, \omega_S)$ varies very fast in time and it is reasonable to approximate $u_S= K\ast \omega_S$ by some noise $\eta$ which is white in time. Now, the equations for the large scale variables $(\theta_R, \omega_R)$ have the form of 2D Boussinesq system \eqref{2D-Boussinesq-1} perturbed a multiplicative noise of transport type:
  \begin{equation}\label{stoch-2D-Boussinesq}
  \left\{\aligned
  & \partial_t \theta + u\cdot\nabla \theta = \kappa\Delta \theta + \eta \cdot\nabla \theta,\\
  & \partial_t\omega + u\cdot\nabla \omega = \partial_1\theta + \eta \cdot\nabla \omega, \\
  & u= K\ast \omega,
  \endaligned \right.
  \end{equation}
where we have changed the variables $(\theta_R, \omega_R, u_R)$ back to $(\theta, \omega, u)$. The transport noise can also be motivated by some arguments from variational principles, see \cite{Holm}. We mention that stochastic 2D Navier-Stokes equations perturbed by multiplicative transport noise have already been considered in \cite{BCF, MikRoz}.

The noise $\eta$ considered in this paper has the following form:
  $$\eta(t,x)= \frac{\sqrt{2\nu}}{\|\sigma \|_{\ell^2}} \sum_k \sigma_k e_k(x) \dot W^k_t, $$
where $\nu>0$ represents the intensity of noise and $\sigma= \{\sigma_k \}_{k\in \Z^2_0} $ belongs to $\ell^2:= \ell^2(\Z^2_0)$, the space of square summable real sequences with the norm $\|\sigma \|_{\ell^2}= \big(\sum_k \sigma_k^2 \big)^{1/2}$; here, $\Z^2_0= \Z^2 \setminus \{0 \}$. We shall mainly consider those $\sigma$ with compact support, namely, there are only finitely many $k\in \Z^2_0$ such that $\sigma_k \neq 0$, and assume that $\sigma$ verifies the symmetry property:
  \begin{equation}\label{sigma-symmetry}
  \sigma_k =\sigma_l \quad \mbox{for all } k,l\in \Z^2_0 \mbox{ with } |k|=|l| .
  \end{equation}
The family $\{W^k\}_{k\in \Z^2_0}$ consists of independent standard complex Brownian motions, satisfying
  \begin{equation}\label{joint-variation}
  [W^k, W^l]_t= 2t \delta_{k,-l},\quad k,l\in \Z^2_0,
  \end{equation}
where $[\cdot, \cdot]$ is the joint quadratic variation. Finally, the family of divergence free vector fields $\{e_k\}_{k\in \Z^2_0}$ on $\T^2$ is given by
  $$e_k(x)= e^{2\pi ik\cdot x} \begin{cases}
  \displaystyle \frac{k^\perp}{|k|}, & k\in \Z^2_+ , \\
  \displaystyle -\frac{k^\perp}{|k|}, & k\in \Z^2_-,
  \end{cases}
  \qquad x\in \T^2, $$
where $\Z^2_0= \Z^2_+ \cup \Z^2_-$ is a partition of $\Z^2_0$ such that $\Z^2_+ = -\Z^2_-$. The noise $\eta(t,x)$ has the covariance matrix (use \eqref{joint-variation})
  $$Q(x,y) = \E [\eta(t,x)\otimes \eta(t,y)] =  \frac{2\nu}{\|\sigma \|_{\ell^2}^2} \sum_k \sigma_k^2 \frac{k^\perp\otimes k^\perp}{|k|^2} e^{2\pi ik\cdot (x-y)}; $$
in particular,
  \begin{equation}\label{convariance}
  Q(x,x) = \frac{2\nu}{\|\sigma \|_{\ell^2}^2} \sum_k \sigma_k^2 \frac{k^\perp\otimes k^\perp}{|k|^2} = \nu I_2,
  \end{equation}
where $I_2$ is the unit matrix of order 2.

Now, the stochastic equations considered in this paper can be written more explicitly as
  \begin{equation}\label{stoch-2D-Boussinesq-1}
  \left\{\aligned
  & \d \theta + u\cdot\nabla \theta\,\d t = \kappa\Delta \theta \,\d t + \frac{\sqrt{2\nu}}{\|\sigma \|_{\ell^2}} \sum_k \sigma_k e_k \cdot\nabla \theta\circ \d W^k_t,\\
  & d\omega + u\cdot\nabla \omega\,\d t = \partial_1\theta\,\d t + \frac{\sqrt{2\nu}}{\|\sigma \|_{\ell^2}} \sum_k \sigma_k e_k \cdot\nabla \omega\circ \d W^k_t,
  \endaligned \right.
  \end{equation}
where $u= K\ast \omega$ and $\circ$ means that the stochastic differential is understood in the Stratonovich sense. When the diffusion coefficient $\kappa=0$, the local existence of regular solutions to \eqref{stoch-2D-Boussinesq-1} was proved in \cite{Alonson-Leon} under suitable conditions on the noise. By \eqref{convariance}, it is not difficult to show that the equations \eqref{stoch-2D-Boussinesq-1} have the It\^o form below:
  \begin{equation}\label{stoch-2D-Boussinesq-Ito}
  \left\{\aligned
  & \d \theta + u\cdot\nabla \theta\,\d t = (\kappa+\nu)\Delta \theta \,\d t + \frac{\sqrt{2\nu}}{\|\sigma \|_{\ell^2}} \sum_k \sigma_k e_k \cdot\nabla \theta\, \d W^k_t,\\
  & d\omega + u\cdot\nabla \omega\,\d t = (\partial_1\theta + \nu \Delta\omega)\,\d t + \frac{\sqrt{2\nu}}{\|\sigma \|_{\ell^2}} \sum_k \sigma_k e_k \cdot\nabla \omega\, \d W^k_t,
  \endaligned \right.
  \end{equation}
with $ u= K\ast \omega$. We remark that, although the Laplacian operator appears in the vorticity equation above, the second equation in \eqref{stoch-2D-Boussinesq-Ito} is not dissipative since it is equivalent to the original Stratonovich equation in \eqref{stoch-2D-Boussinesq-1}.

\begin{remark}
It was proved in \cite{Danchin-Paicu} that \eqref{2D-Boussinesq-1} is globally well posed for Yudovich type initial data, namely, the vorticity $\omega_0$ is bounded. It would be nice to extend this result to the stochastic setting, namely, proving that \eqref{stoch-2D-Boussinesq-1} admits a unique global solution for bounded initial vorticity, similarly to \cite{BFM} which extends Yudovich's result for 2D Euler equations to the stochastic case.
\end{remark}

As in the deterministic case, the system \eqref{stoch-2D-Boussinesq-Ito} with $L^2$-initial data has a weak solution $(\theta, \omega)$ such that, $\P$-a.s., $\theta\in L^\infty \big(0,T; L^2(\T^2) \big) \cap L^2 \big(0,T; H^1(\T^2) \big)$ and $\omega\in L^\infty \big(0,T; L^2(\T^2)\big)$. The solution is weak in both analytic sense and probabilistic sense; see Section 2 for the definition and existence of solutions to \eqref{stoch-2D-Boussinesq-Ito}. Motivated by the recent papers \cite{FlaLuoAoP, Galeati, FGL, LuoSaal}, we will prove that, for any suitably chosen sequence $\{\sigma^N \}_{N\geq 1} \subset \ell^2$, the martingale part in the stochastic 2D Boussinesq system \eqref{stoch-2D-Boussinesq-Ito} will vanish in a scaling limit and what we obtain is the deterministic system of viscous Boussinesq equations:
  \begin{equation}\label{viscid-2D-Boussinesq}
  \left\{\aligned
  & \partial_t\theta + u\cdot\nabla \theta = (\kappa+\nu)\Delta \theta,\\
  & \partial_t \omega + u\cdot\nabla \omega = \partial_1\theta + \nu \Delta\omega,
  \endaligned \right.
  \end{equation}
where $ u= K\ast \omega$. It is well known that, for any initial data $(\theta_0,\omega_0) \in (L^2(\T^2))^2$, the system \eqref{viscid-2D-Boussinesq} admits a unique global solution, see e.g. Theorem \ref{thm-uniqueness} for a proof.

Before stating the main result of this paper, we introduce some notations. In the sequel, $L^2(\T^2)$ and, more generally, the Sobolev spaces $H^s(\T^2)\, (s\in \R)$ are assumed to consisting of functions with zero average; sometimes they will be simply written as $L^2$ and $H^s$. $\|\cdot \|_{L^2}$ and $\|\cdot \|_{H^s}$ will be the norms in $L^2(\T^2)$ and $H^s(\T^2)\, (s\in \R)$, respectively. For $(\theta_0,\omega_0) \in (L^2(\T^2))^2$, define $\|(\theta_0,\omega_0) \|_{L^2} = \|\theta_0 \|_{L^2} \vee \|\omega_0 \|_{L^2}$; moreover, given $\sigma \in \ell^2$, we denote by
  $$\mathcal C_\sigma(\theta_0, \omega_0) = \mbox{collection of laws of weak solutions } (\theta, \omega) \mbox{ to \eqref{stoch-2D-Boussinesq-Ito} with initial data } (\theta_0,\omega_0).$$
Moreover, we define the space
  \begin{equation}\label{space-support}
  \mathcal X= \big( L^2(0,T, L^2(\T^2)) \cap C(0,T, H^-(\T^2)) \big) \times C(0,T, H^-(\T^2))
  \end{equation}
endowed with the usual norm $\|\cdot \|_{\mathcal X}$ of product space. Here, $H^-(\T^2)=\cap_{s<0} H^s(\T^2)$.  Finally, we denote by $\Phi_\cdot (\theta_0,\omega_0)$ the unique solution to \eqref{viscid-2D-Boussinesq} with initial data $(\theta_0,\omega_0) \in (L^2(\T^2))^2$.

\begin{theorem}[Scaling limit] \label{main-thm}
Assume that the family $\{\sigma^N \}_{N\geq 1} \subset \ell^2$ satisfies \eqref{sigma-symmetry} and
  \begin{equation}\label{main-thm.1}
  \lim_{N\to \infty} \frac{\|\sigma^N\|_{\ell^\infty}}{\|\sigma^N\|_{\ell^2}} =0,
  \end{equation}
where $\|\sigma^N\|_{\ell^\infty}= \sup_{k} |\theta^N_k|$. Then, for any $R>0$ and any $\eps>0$, we have
  $$\lim_{N\to \infty} \sup_{\|(\theta_0,\omega_0) \|_{L^2} \leq R} \sup_{Q \in \mathcal C_{\sigma^N}(\theta_0, \omega_0)} Q \big(\{\varphi \in \mathcal X: \|\varphi - \Phi_\cdot (\theta_0,\omega_0) \|_{\mathcal X} >\eps \} \big) =0. $$
\end{theorem}

A simple example for the condition \eqref{main-thm.1} is given here: for a constant $\beta\in [0,\infty)$, let
  $$\sigma^N_k = \begin{cases}
  \displaystyle \frac1{|k|^\beta} \textbf{1}_{\{|k|\leq N\}}, & \mbox{if } \beta\in [0,1], \\
  \displaystyle \frac1{|k|^\beta} \textbf{1}_{\{N\leq |k|\leq 2N\}}, & \mbox{if } \beta\in (1,\infty),
  \end{cases}
  \qquad k\in \Z^2_0. $$
The reason for the difference in definitions is that $\sum_k \frac1{|k|^{2\beta}} <\infty$ if $\beta>1$, thus we need $\|\sigma^N \|_{\ell^\infty} \to 0$ as $N\to \infty$, which is ensured by setting $\sigma^N_k =0$ for all $|k|< N$. Theorem \ref{main-thm} will be proved in Section 3.

\begin{remark}
\begin{itemize}
\item[(i)] Following \cite{LuoSaal}, one can generalize the result of  Theorem \ref{main-thm} to Boussinesq type equations:
  \begin{equation*}
  \left\{\aligned
  & \d \theta + u\cdot\nabla \theta\,\d t = \kappa\Delta \theta \,\d t + \frac{\sqrt{2\nu}}{\|\sigma \|_{\ell^2}} \sum_k \sigma_k e_k \cdot\nabla \theta\circ \d W^k_t,\\
  & d\omega + u\cdot\nabla \omega\,\d t = \partial_1\theta\,\d t + \frac{\sqrt{2\nu}}{\|\sigma \|_{\ell^2}} \sum_k \sigma_k e_k \cdot\nabla \omega\circ \d W^k_t, \\
  & u= K_\eps \ast \omega,
  \endaligned \right.
  \end{equation*}
where, for some $\eps\in (0,1)$, $K_\eps$ is the kernel corresponding to the operator $-\nabla^\perp (-\Delta)^{-(1+\eps)/2}$, and $K_1$ coincides with the Biot-Savart kernel. The limit equations will be
  $$ \left\{\aligned
  & \partial_t\theta + u\cdot\nabla \theta = (\kappa+\nu) \Delta \theta,\\
  & \partial_t \omega + u\cdot\nabla \omega = \partial_1\theta + \nu \Delta\omega,\\
  & u= K_\eps \ast \omega.
  \endaligned \right. $$
\item[(ii)] On the contrary, it is not clear to the author how to treat the system \eqref{stoch-2D-Boussinesq-1} with zero diffusivity, namely, $\kappa =0$.
\end{itemize}
\end{remark}

The above limit result has some interesting consequences which have been briefly discussed at the beginning of the introduction; here, we explain in more detail the approximate weak uniqueness of the stochastic 2D Boussinesq system \eqref{stoch-2D-Boussinesq-1}. To the author's knowledge, for $L^2$-initial data, it is still an open question whether the uniqueness of solutions holds for the deterministic 2D Boussinesq system \eqref{2D-Boussinesq-1}. Although we cannot prove that transport noises have enough regularizing effect so that the stochastic 2D Boussinesq equations \eqref{stoch-2D-Boussinesq-1} admit a unique solution in the $L^2$-setting, the above theorem gives us a uniquely solvable system in the limit. As a result, distances between the laws of weak solutions to \eqref{stoch-2D-Boussinesq-1} will vanish as $N\to \infty$. To be more precise, recall that $\mathcal C_{\sigma^N}(\theta_0, \omega_0)$ is the collection of laws $Q_N$ of weak solutions $(\theta_N, \omega_N)$ to \eqref{stoch-2D-Boussinesq-1} with initial data $(\theta_0, \omega_0) \in (L^2(\T^2))^2$, which can be viewed as probability measures on $\mathcal X$ defined in \eqref{space-support}. Let $\rho(\cdot, \cdot)$ be a distance on $\mathcal X$ metrizing the weak convergence of probability measures on $\mathcal X$. Then, from Theorem \ref{main-thm} it is not difficult to deduce that
  $$\sup_{Q_N, Q'_N \in \mathcal C_{\sigma^N}(\theta_0, \omega_0)} \rho(Q_N, Q'_N) \to 0 \quad \mbox{as } N\to \infty. $$
Therefore, it is reasonable to say that transport noises approximatively regularize the 2D Boussinesq equations.

The paper is organized as follows. In Section 2, we give the meaning of weak solutions to \eqref{stoch-2D-Boussinesq-Ito} and provide a relatively sketched proof of existence of solutions. Section 3 is devoted to the proof of Theorem \ref{main-thm}, by following the method of \cite[Section 3]{LuoSaal}. Finally, for the reader's convenience, we give in Section 4 a proof of uniqueness of solutions to the deterministic viscous Boussinesq system \eqref{viscid-2D-Boussinesq}.

\section{Existence of weak solutions to \eqref{stoch-2D-Boussinesq-Ito}}

In this section we show that, for any $L^2$-initial data $(\theta_0, \omega_0)$, the stochastic 2D Boussinesq equations \eqref{stoch-2D-Boussinesq-Ito} have a weak solution with some desired regularity properties. First, we introduce a few more notations. $C^\infty(\T^2)$ is the space of smooth functions on $\T^2$.  We write $\<\cdot, \cdot\>$ for the inner product in $L^2(\T^2)$. For $p, q\in [1,\infty]$, we shall write $\|\cdot \|_{L^p(L^q)}$ for the norm in $L^p(0,T; L^q(\T^2))$. Recall that the Biot-Savart kernel $K$ can be viewed as a bounded linear operator from $H^s(\T^2)$ to $H^{s+1}(\T^2)$. In the sequel we use $\Xi$ to denote probability spaces since the notation $\omega$ has been reserved for vorticity.

We first explain the meaning of weak solutions to \eqref{stoch-2D-Boussinesq-Ito}.

\begin{definition}\label{2-def}
Given $(\theta_0, \omega_0)\in ( L^2(\T^2))^2$, we say that the stochastic 2D Boussinesq system \eqref{stoch-2D-Boussinesq-Ito} has a weak solution with initial data $(\theta_0, \omega_0)$ if there exists a filtered probability $(\Xi, \F, (\F_t),\P)$, a family of independent $(\F_t)$-Brownian motions $\{W^k_t \}_{k\in \Z^2_0}$, and a pair of $(\F_t)$-progressively measurable processes $\{(\theta_t, \omega_t) \}_{t\in [0,T]}$  such that
\begin{itemize}
\item[\rm(a)] one has $\theta, \omega \in L^2\big(\Xi, L^2(0,T; L^2(\T^2)) \big)$;

\item[\rm(b)] for any $\phi\in C^\infty(\T^2)$, $\P$-a.s. for all $t\in [0,T]$, the following equalities hold:
  \begin{equation}\label{2-def.1}
  \aligned
  \<\theta_t,\phi\> &= \<\theta_0,\phi\> + \int_0^t \big[ \<\theta_s, u_s\cdot \nabla\phi \> + (\kappa+\nu)\<\theta_s, \Delta \phi \> \big]\,\d s \\
  &\quad - \frac{\sqrt{2\nu}}{\|\sigma \|_{\ell^2}} \sum_k \sigma_k \int_0^t \<\theta_s, e_k\cdot \nabla\phi \>\,\d W^k_s,
  \endaligned
  \end{equation}
  \begin{equation}\label{2-def.2}
  \aligned
  \<\omega_t,\phi\> &= \<\omega_0,\phi\> + \int_0^t \big[ \<\omega_s, u_s\cdot \nabla\phi \> - \<\theta_s, \partial_1\phi\> + \nu\<\omega_s, \Delta \phi \> \big]\,\d s \\
  &\quad - \frac{\sqrt{2\nu}}{\|\sigma \|_{\ell^2}} \sum_k \sigma_k \int_0^t \<\omega_s, e_k\cdot \nabla\phi \>\,\d W^k_s,
  \endaligned
  \end{equation}
where $u_s= K\ast \omega_s,\, s\in [0,T]$.
\end{itemize}
\end{definition}

It is clear that, thanks to the properties in (a), all the terms in \eqref{2-def.1} and \eqref{2-def.2} make sense. For instance, by the It\^o isometry and \eqref{joint-variation},
  $$\aligned
  \E\bigg(\Big| \frac{\sqrt{2\nu}}{\|\sigma \|_{\ell^2}} \sum_k \sigma_k \int_0^t \<\theta_s, e_k\cdot \nabla\phi \>\,\d W^k_s \Big|^2 \bigg) &= \frac{4\nu}{\|\sigma \|_{\ell^2}^2} \E \bigg(\sum_k \sigma_k^2 \int_0^t |\<\theta_s, e_k\cdot \nabla\phi \>|^2 \,\d s \bigg) \\
  &\leq \frac{4\nu}{\|\sigma \|_{\ell^2}^2} \sum_k \sigma_k^2\, \E\int_0^t \|\theta_s\|_{L^2}^2 \|e_k\cdot \nabla\phi \|_{L^2}^2 \,\d s \\
  &\leq 4\nu \|\nabla\phi\|_{L^2}^2\, \E\int_0^T \|\theta_s\|_{L^2}^2 \,\d s <\infty,
  \endaligned $$
which implies that the martingale part in \eqref{2-def.1} is square integrable. Similar result holds for \eqref{2-def.2}. The main result of this section gives the existence of weak solutions to \eqref{stoch-2D-Boussinesq-Ito}.

\begin{theorem}\label{thm-existence}
For any $(\theta_0, \omega_0)\in ( L^2(\T^2))^2$, the stochastic 2D Boussinesq system \eqref{stoch-2D-Boussinesq-Ito} admits a weak solution $(\theta,\omega)$ in the sense of Definition \ref{2-def}. Moreover, the processes $\theta$ and $\omega$ have trajectories in $L^\infty(0,T; L^2)\cap L^2(0,T; H^1)$ and in $L^\infty(0,T; L^2)$, respectively, and
  $$ \|\theta \|_{L^\infty(L^2)} \vee \|\nabla \theta \|_{L^2(L^2)}\leq (1\vee \kappa^{-1/2}) \|\theta_0 \|_{L^2},\quad \|\omega \|_{L^\infty(L^2)} \leq \|\theta_0 \|_{L^2} + C_{\kappa, T} \|\omega_0 \|_{L^2}, $$
where $C_{\kappa, T}$ is some constant depending only on $\kappa$ and $T$.
\end{theorem}

The rest of this section is devoted to the proof of Theorem \ref{thm-existence}. The arguments are by now classical, and make use of the Galerkin approximation and compactness method, see e.g. \cite{FlaGat95, FGL}. First we give some a priori estimates.

\begin{lemma}\label{lem-apriori}
The following estimates hold $\P$-a.s. for all $t\in [0,T]$:
  $$\|\theta_t \|_{L^2}^2 + \kappa \int_0^t \|\nabla \theta_s \|_{L^2}^2\,\d s \leq \|\theta_0 \|_{L^2}^2, \quad \|\omega_t \|_{L^2} \leq \|\omega_0 \|_{L^2} + C_{\kappa, T} \|\theta_0 \|_{L^2},$$
where $C_{\kappa, T}$ is some constant depending only on $\kappa, T$.
\end{lemma}

\begin{proof}
Recall the first equation in \eqref{stoch-2D-Boussinesq-Ito}; by the It\^o formula and \eqref{joint-variation},
  $$\aligned
  \d \|\theta \|_{L^2}^2 &= 2\<\theta, -u\cdot\nabla\theta + (\kappa +\nu) \Delta\theta \>\,\d t +\frac{2\sqrt{2\nu}}{\|\sigma \|_{\ell^2}} \sum_k \sigma_k\<\theta, e_k \cdot\nabla \theta\>\, \d W^k_t \\
  &\quad + \frac{4\nu}{\|\sigma \|_{\ell^2}^2} \sum_k \sigma_k^2 \|e_k \cdot \nabla\theta \|_{L^2}^2\,\d t \\
  &= -2(\kappa +\nu) \|\nabla \theta \|_{L^2}^2\,\d t + \frac{4\nu}{\|\sigma \|_{\ell^2}^2} \sum_k \sigma_k^2 \|e_k \cdot \nabla\theta \|_{L^2}^2\,\d t,
  \endaligned $$
where in the second step we have used integration by parts and the facts that $\nabla\cdot u = \nabla\cdot e_k=0$. By a similar computation as in \eqref{convariance}, we have
  $$ \frac{4\nu}{\|\sigma \|_{\ell^2}^2} \sum_k \sigma_k^2 \|e_k \cdot \nabla\theta \|_{L^2}^2=  \frac{4\nu}{\|\sigma \|_{\ell^2}^2} \|\sigma \|_{\ell^2}^2 \|\nabla\theta \|_{L^2}^2 =2\nu \|\nabla\theta \|_{L^2}^2, $$
and thus $\d \|\theta \|_{L^2}^2 = -2\kappa \|\nabla \theta \|_{L^2}^2\,\d t$. This immediately gives us the first estimate.

In the same way, using the second equation in  \eqref{stoch-2D-Boussinesq-Ito}, we have
  $$\aligned
  \d \|\omega \|_{L^2}^2 &= 2\<\omega, -u\cdot\nabla\omega + \partial_1 \theta + \nu \Delta\omega\>\,\d t +\frac{2\sqrt{2\nu}}{\|\sigma \|_{\ell^2}} \sum_k \sigma_k\<\omega, e_k \cdot\nabla \omega\>\, \d W^k_t \\
  &\quad + \frac{4\nu}{\|\sigma \|_{\ell^2}^2} \sum_k \sigma_k^2 \|e_k \cdot \nabla\omega \|_{L^2}^2\,\d t \\
  &= 2\<\omega, \partial_1 \theta \>\,\d t.
  \endaligned $$
The Cauchy inequality leads to
  $$\d \|\omega \|_{L^2}^2 \leq 2 \|\omega \|_{L^2} \|\partial_1 \theta \|_{L^2}\,\d t\leq 2 \|\omega \|_{L^2} \|\nabla \theta \|_{L^2}\,\d t , $$
and thus $\d \|\omega \|_{L^2} \leq \|\nabla \theta \|_{L^2}\,\d t$. Therefore,
  $$\|\omega_t \|_{L^2} \leq \|\omega_0 \|_{L^2} + \int_0^t \|\nabla \theta_s \|_{L^2}\,\d s \leq \|\omega_0 \|_{L^2} +  \sqrt{T} \|\nabla \theta \|_{L^2(L^2)}, $$
where $\|\cdot \|_{L^2(L^2)}$ is the norm in $L^2\big(0,T; L^2(\T^2)\big)$. Combining this inequality with the first estimate, we complete the proof.
\end{proof}

Let $H_N$ be the finite dimensional subspace of $L^2(\T^2)$ spanned by $e^{2\pi ik\cdot x}, 0<|k|\leq N$ and $\Pi_N:L^2(\T^2) \to H_N$ the orthogonal projection. We consider the Galerkin approximation of the system \eqref{stoch-2D-Boussinesq-Ito}:
  \begin{equation}\label{Boussinesq-Galerkin}
  \left\{ \aligned
  & \d \theta^N + \Pi_N(u^N\cdot \nabla \theta^N) \,\d t = (\kappa + \nu)\Delta \theta^N \,\d t + \frac{\sqrt{2\nu}}{\|\sigma \|_{\ell^2}} \sum_k \sigma_k \Pi_N(e_k \cdot\nabla \theta^N) \, \d W^k_t, \\
  & \d \omega^N + \Pi_N(u^N\cdot \nabla \omega^N) \,\d t = (\partial_1 \theta^N + \nu\Delta \omega^N)\,\d t + \frac{\sqrt{2\nu}}{\|\sigma \|_{\ell^2}} \sum_k \sigma_k \Pi_N(e_k \cdot\nabla \omega^N) \, \d W^k_t
  \endaligned \right.
  \end{equation}
with the initial data $\theta^N_0 = \Pi_N\theta_0$ and $\omega^N_0 = \Pi_N\omega_0$. Here, $u^N= K\ast \omega^N$. Using Lemma \ref{lem-apriori}, we obtain the following uniform bounds: for all $N\geq 1$, $\P$-a.s. for all $t\in [0,T]$,
  \begin{equation}\label{bounds-Galerkin}
  \|\theta^N_t \|_{L^2}^2 + \kappa \int_0^t \|\nabla \theta^N_s \|_{L^2}^2\,\d s \leq \|\theta_0 \|_{L^2}^2, \quad \|\omega^N_t \|_{L^2} \leq \|\omega_0 \|_{L^2} + C_{\kappa, T} \|\theta_0 \|_{L^2},
  \end{equation}
where $C_{\kappa, T}$ is independent of $N$. It follows that the family $\{(\theta^N, \omega^N) \}_{N\geq 1}$ is bounded in
  $$ L^\infty\big( \Xi, L^\infty(0,T, L^2(\T^2)) \cap L^2(0,T, H^1(\T^2))\big) \times L^\infty\big( \Xi, L^\infty(0,T, L^2(\T^2))\big). $$
As a result, we can find a subsequence $\{(\theta^{N_i}, \omega^{N_i}) \}_{i\geq 1}$ which is weakly convergent in the above space. In order to show that the weak limit solves the equations \eqref{2-def.1} and \eqref{2-def.2}, we need stronger convergence result.

Let $\eta_N$ be the joint law of $(\theta^N, \omega^N),\, N\geq 1$; we want to show that the family $\{\eta_N \}_{N\geq 1}$ is tight on the following space
  $$\big( L^2(0,T, L^2(\T^2)) \cap C(0,T, H^-(\T^2)) \big) \times C(0,T, H^-(\T^2)). $$
Denote by $P_N$ (resp. $Q_N$) the law of $\theta^N$ (resp. $\omega^N$), $N\geq 1$; it is sufficient to show that the two families $\{P_N \}_{N\geq 1}$ and $\{Q_N \}_{N\geq 1}$ are tight respectively on $L^2(0,T, L^2(\T^2)) \cap C(0,T, H^-(\T^2))$ and on $C(0,T, H^-(\T^2))$. For this purpose, we will apply Simon's compactness theorems (cf. \cite{Simon}) which makes use of time fractional Sobolev space. For $\gamma\in (0,1)$, $p>1$ and a normed linear space $(Y, \|\cdot \|_Y)$, the fractional Sobolev space $W^{\gamma, p}(0,T; Y)$ consists of those functions $\varphi\in L^p(0,T; Y)$ such that
  $$\int_0^T\!\int_0^T \frac{\|\varphi(t)- \varphi(s)\|_{Y}^p}{|t-s|^{1+\gamma p}}\,\d t\d s <\infty. $$
In the following we take $Y= H^{-\alpha}$ with $\alpha>5$, a choice due to the calculations in the proof of Corollary \ref{cor-tightness-2}.

\begin{proposition}\label{prop-Simon}
\begin{itemize}
\item[\rm(i)] For any $\gamma\in (0, 1/2)$, we have the compact embedding:
  $$L^2\big(0,T; H^1 \big) \cap W^{\gamma,2} \big(0,T; H^{-\alpha} \big) \subset L^2\big(0,T; L^2\big).$$
\item[\rm(ii)] The following embedding is compact:
  $$L^\infty \big(0,T; L^2\big) \cap W^{1/3,4} \big(0,T; H^{-\alpha} \big) \subset C([0,T], H^- ). $$
\end{itemize}
\end{proposition}

\begin{proof}
Assertion (i) is a direct consequence of \cite[page 86, Corollary 5]{Simon}. Next, we deduce from \cite[page 90, Corollary 9]{Simon} that, for any fixed $\delta>0$, the inclusion
  $$L^p \big(0,T; L^2\big) \cap W^{1/3,4} \big(0,T; H^{-\alpha} \big) \subset C\big([0,T], H^{-\delta} \big)$$
is compact for all $p$ big enough. This implies the second assertion since $\delta>0$ is arbitrary.
\end{proof}

For $N\geq 1$, let $\xi_N$ be a stochastic process with trajectories in $L^\infty(0,T; L^2) \cap L^\infty(0,T; H^1)$ and denote its law by $L_N$. One can deduce the following results from Proposition \ref{prop-Simon}.

\begin{corollary}\label{cor-tightness}
\begin{itemize}
\item[\rm(i)] If there is $C>0$ such that for all $N\geq 1$,
  \begin{equation}\label{cor-tightness.1}
  \E\int_0^T \|\xi_N(t) \|_{H^1}^2\,\d t + \E \int_0^T\!\int_0^T \frac{\|\xi_N(t)- \xi_N(s) \|_{H^{-\alpha}}^2}{|t-s|^{1+2\gamma}} \,\d t\d s \leq C,
  \end{equation}
then $\{L_N \}_{N\in \N}$ is tight on $L^2(0,T; L^2)$.
\item[\rm(ii)] If for any $p>1$, there is $C_p>0$ such that for all $N\geq 1$,
  \begin{equation}\label{cor-tightness.2}
  \E\int_0^T \|\xi_N(t) \|_{L^2}^p \,\d t + \E \int_0^T\!\int_0^T \frac{\|\xi_N(t)- \xi_N(s) \|_{H^{-\alpha}}^4 }{|t-s|^{7/3}} \,\d t\d s \leq C_p ,
  \end{equation}
then $\{L_N \}_{N\in \N}$ is tight on $C([0,T], H^{-} )$.
\end{itemize}
\end{corollary}

In order to apply Corollary \ref{cor-tightness} to the laws of the processes $\{\theta^N \}_{N\geq 1}$ and $\{\omega^N \}_{N\geq 1}$, we need to estimate the $H^{-4}$-norm of the increments $\theta^N_t -\theta^N_s$ and $\omega^N_t -\omega^N_s$, $0\leq s<t \leq T$. This is done in the following lemma.

\begin{lemma}\label{lem-increments}
Denote by $f_k(x)= e^{2\pi i k\cdot x},\, x\in \T^2, k\in \Z^2_0$. There exists $C= C_{\kappa, \nu, T, \|\theta_0 \|_{L^2}, \|\omega_0 \|_{L^2}}>0$, independent of $N$, such that for any $k\in \Z^2_0$ and $0\leq s< t\leq T$, it holds
  $$\E\big(|\<\theta^N_t -\theta^N_s, f_k \>|^4 \big) \vee \E\big(|\<\omega^N_t -\omega^N_s, f_k \>|^4 \big) \leq C|k|^8 (t-s)^2 . $$
\end{lemma}

\begin{proof}
It suffices to consider $k\in \Z^2_0$ with $|k|\leq N$. By the first equation in \eqref{Boussinesq-Galerkin},
  $$\<\theta^N_t -\theta^N_s, f_k \> = \int_s^t \big[ \<\theta^N_r, u^N_r\cdot \nabla f_k \> + (\kappa+\nu)\<\theta^N_r, \Delta f_k \> \big]\,\d r - \frac{\sqrt{2\nu}}{\|\sigma \|_{\ell^2}} \sum_l \sigma_l \int_s^t \<\theta^N_r, e_l \cdot\nabla f_k\> \, \d W^l_r.$$
Therefore,
  $$\aligned
  \E\big(|\<\theta^N_t -\theta^N_s, f_k \>|^4 \big) &\leq C \E \bigg[ \Big|\int_s^t \big[ \<\theta^N_r, u^N_r\cdot \nabla f_k \> + (\kappa+\nu)\<\theta^N_r, \Delta f_k \> \big]\,\d r \Big|^4\bigg]\\
  &\quad + C \E \bigg[ \Big|\frac{\sqrt{2\nu}}{\|\sigma \|_{\ell^2}} \sum_l \sigma_l \int_s^t \<\theta^N_r, e_l \cdot\nabla f_k\> \, \d W^l_r \Big|^4\bigg].
  \endaligned $$
Using the bounds \eqref{bounds-Galerkin} and the facts $|\nabla f_k| = 2\pi |k|,\, |\Delta f_k| = 4\pi^2 |k|^2$, it is easy to show that
  $$\E \bigg[ \Big|\int_s^t \big[ \<\theta^N_r, u^N_r\cdot \nabla f_k \> + (\kappa+\nu)\<\theta^N_r, \Delta f_k \> \big]\,\d r \Big|^4\bigg] \leq C_{\kappa, \nu, T, \|\theta_0 \|_{L^2}, \|\omega_0 \|_{L^2}} |k|^8 (t-s)^4. $$
Next, by the Burkholder-Daves-Gundy inequality and the estimate $\|\theta^N_r\|_{L^2} \leq \|\theta_0\|_{L^2}$,
  $$\aligned
  \E \bigg[ \Big|\frac{\sqrt{2\nu}}{\|\sigma \|_{\ell^2}} \sum_l \sigma_l \int_s^t \<\theta^N_r, e_l \cdot\nabla f_k\> \, \d W^l_r \Big|^4\bigg] &\leq C \frac{\nu^2}{\|\sigma \|_{\ell^2}^4} \E \bigg[ \Big( \sum_l \sigma_l^2 \int_s^t |\<\theta^N_r, e_l \cdot\nabla f_k\>|^2 \, \d r \Big)^2\bigg] \\
  & \leq C \frac{\nu^2}{\|\sigma \|_{\ell^2}^4} \Big( \sum_l \sigma_l^2 \int_s^t \|\theta_0\|_{L^2}^2 \cdot 4\pi^2 |k|^2 \, \d r \Big)^2 \\
  &\leq C_{\nu, \|\theta_0\|_{L^2}} |k|^4 (t-s)^2.
  \endaligned $$
Summing up these arguments we obtain the estimate on the process $\theta^N$. In the same way, using the second equation in \eqref{Boussinesq-Galerkin} we can prove the estimate for $\omega^N$.
\end{proof}

Recall that $P_N$ (resp. $Q_N$) is the law of the process $\theta^N$ (resp. $\omega^N$), $N\geq 1$. With the above estimates in hand, we can prove the tightness of the laws $\{P_N\}_{N\geq 1}$ and  $\{Q_N\}_{N\geq 1}$.

\begin{corollary}\label{cor-tightness-2}
The family of laws $\{P_N\}_{N\geq 1}$ is tight on $ L^2(0,T, L^2(\T^2)) \cap C(0,T, H^-(\T^2)) $ and the family $\{Q_N\}_{N\geq 1}$ is tight on $C(0,T, H^-(\T^2)) $.
\end{corollary}

\begin{proof}
We only show the tightness of $\{P_N\}_{N\geq 1}$; the proof of the second assertion is easier. First, to show the tightness of $\{P_N\}_{N\geq 1}$ on $L^2(0,T, L^2(\T^2))$, by \eqref{bounds-Galerkin} and Corollary \ref{cor-tightness}(i), it suffices to consider the second expectation in \eqref{cor-tightness.1}. We have, by Lemma \ref{lem-increments},
  $$\aligned
  \E\big(\|\theta^N_t -\theta^N_s\|_{H^{-\alpha}}^2 \big) & = \E\Big(\sum_k \frac1{|k|^{2\alpha}} |\<\theta^N_t -\theta^N_s, f_k\>|^2 \Big) \leq \sum_k \frac1{|k|^{2\alpha}} C|k|^4 |t-s| \leq C_\alpha |t-s|
  \endaligned $$
since $\alpha>5$. Here, the constant $C_\alpha$ is independent of $N$. Therefore, for any $N\geq 1$,
  $$\E \int_0^T\!\int_0^T \frac{\|\theta^N_t- \theta^N_s \|_{H^{-\alpha}}^2}{|t-s|^{1+2\gamma}} \,\d t\d s \leq \int_0^T\!\int_0^T \frac{C_\alpha |t-s|}{|t-s|^{1+2\gamma}} \,\d t\d s \leq C'$$
since $\gamma<1/2$. This implies that $\{P_N\}_{N\geq 1}$ is tight on $L^2(0,T, L^2(\T^2))$.

Next, by Cauchy's inequality,
  $$\aligned
  \E\big(\|\theta^N_t -\theta^N_s\|_{H^{-\alpha}}^4 \big)&= \E\bigg[ \Big(\sum_k \frac1{|k|^{2\alpha}} |\<\theta^N_t -\theta^N_s, f_k\>|^2 \Big)^2 \bigg] \\
  &\leq \Big(\sum_k \frac1{|k|^{2\alpha}} \Big) \sum_k \frac1{|k|^{2\alpha}} \E \big(|\<\theta^N_t -\theta^N_s, f_k\>|^4 \big) \\
  &\leq C_\alpha \sum_k \frac1{|k|^{2\alpha}} C|k|^8 |t-s|^2 \leq C'_\alpha |t-s|^2,
  \endaligned $$
where we have used again the fact that $\alpha>5$. As a result,
  $$\E \int_0^T\!\int_0^T \frac{\|\theta^N_t- \theta^N_s \|_{H^{-\alpha}}^4}{|t-s|^{7/3}} \,\d t\d s \leq \int_0^T\!\int_0^T \frac{C'_\alpha |t-s|^2}{|t-s|^{7/3}} \,\d t\d s \leq C''. $$
The assertion (ii) of Corollary \ref{cor-tightness} implies the tightness of $\{P_N\}_{N\geq 1}$ on $C(0,T, H^-(\T^2)) $.
\end{proof}

Now recall that $\eta_N$ is the joint law of the pair of processes $(\theta^N, \omega^N),\, N\geq 1$. We conclude from Corollary \ref{cor-tightness-2} that the family $\{\eta_N \}_{N\geq 1}$ is tight on
  $$\mathcal X= \big( L^2(0,T, L^2(\T^2)) \cap C(0,T, H^-(\T^2)) \big) \times C(0,T, H^-(\T^2)). $$
The rest of the arguments is quite classical, and thus we only give a sketch here. By the Prohorov theorem (\cite[p.59, Theorem
5.1]{Billingsley}), we can find a subsequence $\{N_i \}_{i\geq 1}$ such that $\eta_{N_i}$ converges weakly as $i\to \infty$ to some probability measure $\eta$ supported on $\mathcal X$. Next, the Skorokhod representation theorem (\cite[p.70, Theorem 6.7]{Billingsley}) implies that there exist  a new probability space $\big(\tilde\Xi, \tilde\F, \tilde\P \big)$, a sequence of stochastic processes $\big(\tilde \theta^{N_i}, \tilde\omega^{N_i} \big)$ and a limit process $\big(\tilde \theta, \tilde\omega \big)$ defined on $\big(\tilde\Xi, \tilde\F, \tilde\P \big)$, such that
\begin{itemize}
\item[(1)] for any $i\geq 1$, the pair $\big(\tilde \theta^{N_i}, \tilde\omega^{N_i} \big)$ has the same law $\eta_{N_i}$ as $(\theta^{N_i}, \omega^{N_i} )$;
\item[(2)] $\tilde\P$-a.s., $\big(\tilde \theta^{N_i}, \tilde\omega^{N_i} \big)$ converges in the topology of
  $$\big( L^2(0,T, L^2(\T^2)) \cap C(0,T, H^-(\T^2)) \big) \times C(0,T, H^-(\T^2)) $$
  to the limit process $\big(\tilde \theta, \tilde\omega \big)$.
\end{itemize}

Combining the assertion (1) and the uniform bounds \eqref{bounds-Galerkin}, we conclude that, $\tilde\P$-a.s. for all $i\geq 1$ and $t\in [0,T]$,
  \begin{equation}\label{bounds-new-processes}
  \|\tilde\theta^{N_i}_t \|_{L^2}^2 + \kappa \int_0^t \|\nabla \tilde\theta^{N_i}_s \|_{L^2}^2\,\d s \leq \|\theta_0 \|_{L^2}^2, \quad \|\tilde\omega^{N_i}_t \|_{L^2} \leq \|\omega_0 \|_{L^2} + C_{\kappa, T} \|\theta_0 \|_{L^2}.
  \end{equation}
From these bounds, it is not difficult to show that (cf. \cite[Lemma 3.5]{FGL}) the limit processes $\tilde\theta$ and $\tilde \omega$ satisfy
  $$\tilde\P \mbox{-a.s.}, \quad \|\tilde\theta \|_{L^\infty(L^2)} \leq \|\theta_0 \|_{L^2} \quad \mbox{and} \quad \|\tilde\omega \|_{L^\infty(L^2)} \leq \|\theta_0 \|_{L^2} + C_{\kappa, T} \|\omega_0 \|_{L^2} . $$
Moreover, one can prove that the process $\tilde\theta$ is weakly differentiable in the spatial variable and
  $$\tilde\P \mbox{-a.s.}, \quad \|\nabla \tilde\theta \|_{L^2(L^2)} \leq \frac1{\sqrt \kappa} \|\theta_0 \|_{L^2}. $$
Therefore, the pair $\big(\tilde\theta ,\tilde\omega \big)$ verifies the second assertion in Theorem \ref{thm-existence}.

Next, let $\tilde u =K\ast \tilde \omega$ be the corresponding velocity field on the new probability space $\big(\tilde\Xi, \tilde\F, \tilde\P \big)$. Thanks to the convergence results in item (2) above, it is standard to show that the processes $\tilde\theta$, $\tilde\omega$ and $\tilde u$ satisfy the equations \eqref{2-def.1} and \eqref{2-def.2}. We remark that, for any $i\geq 1$, we need also to find a sequence of independent complex Brownian motions $\big\{\tilde W^{N_i,k} \big\}_{k\in \Z^2_0}$, as well as a family of limit Brownian motions $\big\{\tilde W^k \big\}_{k\in \Z^2_0}$, such that $\tilde \P$-a.s., $\tilde W^{N_i,k}$ converges to $\tilde W^k$ in the topology of $C([0,T],\mathbb C)$ for all $k\in \Z^2_0$. To this end, we can consider the law $\mathcal W$ of the family $\{W^k \}_{k\in \Z^2_0}$ together with $\eta_N,\, N\geq 1$. It is easy to show that the family $\{\eta_N\otimes \mathcal W\}_{N\geq 1}$ of joint laws is tight on some suitable space, and then we apply the Prohorov theorem and the Skorokhod theorem. We omit the details here, see the discussions above (3.8) in \cite{FGL}.

\section{Proof of Theorem \ref{main-thm}}

In this section we take a sequence $\{\sigma^N \}_{N\geq 1} \subset \ell^2$ such that, for all $N\geq 1$
  $$\sigma^N_k = \sigma^N_l \quad \mbox{for all } k,l\in \Z^2_0 \mbox{ with } |k|=|l|, $$
and
  \begin{equation}\label{key-property}
  \lim_{N\to \infty} \frac{\|\sigma^N \|_{\ell^\infty}}{\|\sigma^N \|_{\ell^2}} =0.
  \end{equation}
For any $N\geq 1$, we consider the equations
  \begin{equation}\label{stoch-2D-Boussinesq-approx}
  \left\{\aligned
  & \d \theta^N + u^N\cdot\nabla \theta^N\,\d t = (\kappa+\nu)\Delta \theta^N \,\d t + \frac{\sqrt{2\nu}}{\|\sigma^N \|_{\ell^2}} \sum_k \sigma^N_k e_k \cdot\nabla \theta^N\, \d W^k_t,\\
  & d\omega^N + u^N\cdot\nabla \omega^N\,\d t = (\partial_1\theta^N + \nu \Delta\omega^N)\,\d t + \frac{\sqrt{2\nu}}{\|\sigma^N \|_{\ell^2}} \sum_k \sigma^N_k e_k \cdot\nabla \omega^N\, \d W^k_t
  \endaligned \right.
  \end{equation}
with the initial data $(\theta^N_0, \omega^N_0) \in (L^2(\T^2))^2$. Here, $u^N =K\ast \omega^N$. By Theorem \ref{thm-existence}, the above system admits a weak solution $(\theta^N, \omega^N)$ in the sense of Definition \ref{2-def}, defined on some probability space $(\Xi, \mathcal F, \P)$; furthermore, we have the following assertions:
\begin{itemize}
\item[(a$'$)] for any $N\geq 1$, $\P$-a.s., one has
  \begin{equation}\label{estimates}
  \|\theta^N \|_{L^\infty(L^2)} \vee \|\nabla \theta^N \|_{L^2(L^2)}\leq (1\vee \kappa^{-1/2}) \|\theta^N_0 \|_{L^2},\quad \|\omega^N \|_{L^\infty(L^2)} \leq \|\theta^N_0 \|_{L^2} + C_{\kappa, T} \|\omega^N_0 \|_{L^2};
  \end{equation}
\item[\rm(b$'$)] for any $\phi\in C^\infty(\T^2)$, $\P$-a.s. for all $t\in [0,T]$, the following equalities hold:
  \begin{equation}\label{approx.1}
  \aligned
  \<\theta^N_t,\phi\> &= \<\theta^N_0,\phi\> + \int_0^t \big[ \<\theta^N_s, u^N_s\cdot \nabla\phi \> + (\kappa+\nu)\<\theta^N_s, \Delta \phi \> \big]\,\d s \\
  &\quad - \frac{\sqrt{2\nu}}{\|\sigma^N \|_{\ell^2}} \sum_k \sigma^N_k \int_0^t \<\theta^N_s, e_k\cdot \nabla\phi \>\,\d W^k_s,
  \endaligned
  \end{equation}
  \begin{equation}\label{approx.2}
  \aligned
  \<\omega^N_t,\phi\> &= \<\omega^N_0,\phi\> + \int_0^t \big[ \<\omega^N_s, u^N_s\cdot \nabla\phi \> - \<\theta^N_s, \partial_1\phi\> + \nu\<\omega^N_s, \Delta \phi \> \big]\,\d s \\
  &\quad - \frac{\sqrt{2\nu}}{\|\sigma^N \|_{\ell^2}} \sum_k \sigma^N_k \int_0^t \<\omega^N_s, e_k\cdot \nabla\phi \>\,\d W^k_s.
  \endaligned
  \end{equation}
\end{itemize}
We remark that the processes $(\theta^N, \omega^N)$ might be defined on different probability spaces, but for simplicity we do not distinguish the notations $\Xi, \P, \E $ etc.

We first prove the following intermediate convergence result.

\begin{proposition}\label{prop-weak-convergence}
Assume that the sequence $(\theta^N_0, \omega^N_0)$ converges weakly in $(L^2(\T^2))^2$ to some limit $(\theta_0, \omega_0)$. Then, the sequence $(\theta^N, \omega^N)$ of weak solutions converges weakly to the unique solution of the deterministic viscous Boussinesq system
  \begin{equation}\label{prop-weak-convergence.0}
  \left\{\aligned
  & \partial_t\theta + u\cdot\nabla \theta = (\kappa+\nu)\Delta \theta,\\
  & \partial_t \omega + u\cdot\nabla \omega = \partial_1\theta + \nu \Delta\omega
  \endaligned \right.
  \end{equation}
with initial data $(\theta_0, \omega_0)$.
\end{proposition}

\begin{proof}
There is a constant $C>0$ such that
  \begin{equation}\label{prop-weak-convergence.1}
  \sup_{N\geq 1} \big( \|\theta^N_0 \|_{L^2}\vee \|\omega^N_0 \|_{L^2} \big) \leq C<\infty.
  \end{equation}
Combining this bound with the estimates \eqref{estimates}, we can repeat the arguments in Section 2 to show that the family $\{\eta^N \}_{N\geq 1}$ of laws of the processes $\{(\theta^N, \omega^N) \}_{N\geq 1}$ is tight on
  $$\big( L^2(0,T, L^2(\T^2)) \cap C(0,T, H^-(\T^2)) \big) \times C(0,T, H^-(\T^2)). $$
Consequently, applying the Prohorov theorem and the Skorokhod theorem, we can find a subsequence $\{\eta^{N_i} \}_{i\geq 1}$ and a new probability space $\big(\hat\Xi, \hat\F, \hat\P \big)$, and a sequence of processes $\big\{(\hat\theta^{N_i}, \hat\omega^{N_i}) \big\}_{i\geq 1}$ and a limit process $(\hat\theta, \hat\omega)$ defined on $\big(\hat\Xi, \hat\F, \hat\P \big)$, such that
\begin{itemize}
\item[(1$'$)] for any $i\geq 1$, the pair $\big(\hat \theta^{N_i}, \hat\omega^{N_i} \big)$ has the same law $\eta^{N_i}$ as $(\theta^{N_i}, \omega^{N_i} )$;
\item[(2$'$)] $\hat \P$-a.s., $\big(\hat \theta^{N_i}, \hat\omega^{N_i} \big)$ converges in the topology of
  $$\big( L^2(0,T, L^2(\T^2)) \cap C(0,T, H^-(\T^2)) \big) \times C(0,T, H^-(\T^2)) $$
  to the limit process $\big(\hat \theta, \hat\omega \big)$.
\end{itemize}

Combining the assertion (1$'$) with the equations \eqref{approx.1} and \eqref{approx.2}, we conclude that the pair $\big(\hat \theta^{N_i}, \hat\omega^{N_i} \big)$ satisfies the same equations. Indeed, similarly to the discussions at the end of Section 2, for any $i\geq 1$, there exists a family of independent Brownian motions $\{\hat W^{N_i,k} \}_{k\in \Z^2_0}$ such that for all $i\geq 1$, for any $\phi\in C^\infty(\T^2)$, one has $\hat\P$-a.s. for all $t\in [0,T]$,
  \begin{equation}\label{prop-weak-convergence.2}
  \aligned
  \<\hat\theta^{N_i}_t,\phi\> &= \<\theta^{N_i}_0,\phi\> + \int_0^t \big[ \<\hat\theta^{N_i}_s, \hat u^{N_i}_s\cdot \nabla\phi \> + (\kappa+\nu)\<\hat \theta^{N_i}_s, \Delta \phi \> \big]\,\d s \\
  &\quad - \frac{\sqrt{2\nu}}{\|\sigma^{N_i} \|_{\ell^2}} \sum_k \sigma^{N_i}_k \int_0^t \<\hat \theta^{N_i}_s, e_k\cdot \nabla\phi \>\,\d \hat W^{N_i,k}_s,
  \endaligned
  \end{equation}
  \begin{equation}\label{prop-weak-convergence.3}
  \aligned
  \<\hat\omega^{N_i}_t,\phi\> &= \<\omega^{N_i}_0,\phi\> + \int_0^t \big[ \<\hat\omega^{N_i}_s, \hat u^{N_i}_s\cdot \nabla\phi \> - \<\hat\theta^{N_i}_s, \partial_1\phi\> + \nu\<\hat\omega^{N_i}_s, \Delta \phi \> \big]\,\d s \\
  &\quad - \frac{\sqrt{2\nu}}{\|\sigma^{N_i} \|_{\ell^2}} \sum_k \sigma^{N_i}_k \int_0^t \<\hat\omega^{N_i}_s, e_k\cdot \nabla\phi \>\,\d \hat W^{N_i,k}_s,
  \endaligned
  \end{equation}
where $\hat u^{N_i}_s = K\ast \hat \omega^{N_i}_s,\, s\in [0,T]$. Furthermore, thanks to the bounds \eqref{estimates} and \eqref{prop-weak-convergence.1}, we deduce from assertion (1$'$) that, $\hat\P$-a.s.,
  \begin{equation}\label{prop-weak-convergence.4}
  \sup_{i\geq 1} \Big( \|\hat\theta^{N_i} \|_{L^\infty(L^2)} \vee \|\nabla\hat \theta^{N_i} \|_{L^2(L^2)} \Big) \leq C' <\infty, \quad \sup_{i\geq 1} \|\hat\omega^{N_i} \|_{L^\infty(L^2)} \leq C' <\infty.
  \end{equation}

Combining the above uniform estimates with assertion (2$'$), it is easy to show that all the terms in \eqref{prop-weak-convergence.2} and \eqref{prop-weak-convergence.3}, except the stochastic integrals, converge as $i\to \infty$ to the corresponding limits. In the following, we prove that the stochastic integrals vanish in the mean square sense. Indeed, by the Burkholder-Daves-Gundy inequality,
  $$\aligned
  &\quad \E\bigg[\Big| \frac{\sqrt{2\nu}}{\|\sigma^{N_i} \|_{\ell^2}} \sum_k \sigma^{N_i}_k \int_0^t \<\hat \theta^{N_i}_s, e_k\cdot \nabla\phi \>\,\d \hat W^{N_i,k}_s \Big|^2 \bigg] \\
  &= \frac{2\nu}{\|\sigma^{N_i} \|_{\ell^2}^2} \sum_k \big( \sigma^{N_i}_k \big)^2\, \E \int_0^t \big| \<\hat \theta^{N_i}_s, e_k\cdot \nabla\phi \> \big|^2 \,\d s \\
  &\leq 2\nu \frac{\|\sigma^{N_i} \|_{\ell^\infty}^2}{\|\sigma^{N_i} \|_{\ell^2}^2}\, \E \sum_k \int_0^t \big| \<\hat \theta^{N_i}_s, e_k\cdot \nabla\phi \> \big|^2 \,\d s.
  \endaligned $$
Note that the family of vector fields $\{e_k \}_{k\in \Z^2_0}$ is orthonormal; by the Bessel inequality,
  $$\sum_k \big| \<\hat \theta^{N_i}_s, e_k\cdot \nabla\phi \> \big|^2 \leq \|\hat \theta^{N_i}_s \nabla\phi \|_{L^2}^2 \leq \|\hat \theta^{N_i}_s \|_{L^2}^2 \|\nabla\phi \|_{\infty}^2 \leq C \|\nabla\phi \|_{\infty}^2, $$
where the last step follows from \eqref{prop-weak-convergence.4}. Consequently,
  $$\E\bigg[\Big| \frac{\sqrt{2\nu}}{\|\sigma^{N_i} \|_{\ell^2}} \sum_k \sigma^{N_i}_k \int_0^t \<\hat \theta^{N_i}_s, e_k\cdot \nabla\phi \>\,\d \hat W^{N_i,k}_s \Big|^2 \bigg] \leq C\nu T \|\nabla\phi \|_{\infty}^2 \frac{\|\sigma^{N_i} \|_{\ell^\infty}^2}{\|\sigma^{N_i} \|_{\ell^2}^2} $$
which, by \eqref{key-property}, tends to 0 as $i\to \infty$. In the same way, we can show that the martingale part in \eqref{prop-weak-convergence.3} also vanishes as $i\to \infty$. Therefore, the limit $(\hat\theta, \hat\omega)$ is a weak solution to \eqref{prop-weak-convergence.0}. Since the system \eqref{prop-weak-convergence.0} admits a unique solution for any initial data $(\theta_0, \omega_0)\in (L^2(\T^2))^2$, we conclude that the whole sequence $\{(\hat\theta^N, \hat\omega^N) \}_{N\geq 1}$ converges weakly to $(\hat\theta, \hat\omega)$. This immediately leads to the desired result.
\end{proof}

Now we can prove the main result of this paper.

\begin{proof}[Proof of Theorem \ref{main-thm}]
We argue by contraction. Suppose that there exists an $\eps_0>0$ such that
  $$\limsup_{N\to \infty} \sup_{\|(\theta_0,\omega_0) \|_{L^2} \leq R} \sup_{Q \in \mathcal C_{\sigma^N}(\theta_0, \omega_0)} Q \big(\{\varphi \in \mathcal X: \|\varphi - \Phi_\cdot (\theta_0,\omega_0) \|_{\mathcal X} >\eps_0 \} \big) >0.$$
Then, there is a subsequence $\{N_i \}_{i\geq 1}$ and $(\theta^{N_i}_0,\omega^{N_i}_0) \in (L^2(\T^2))^2$ with $\sup_{i\geq 1} \|(\theta^{N_i}_0, \omega^{N_i}_0) \|_{L^2} \leq R$, and $Q^{N_i} \in \mathcal C_{\sigma^{N_i}}(\theta^{N_i}_0, \omega^{N_i}_0)$ such that (choose $\eps_0$ even smaller if necessary)
  \begin{equation}\label{proof-converg.1}
  Q^{N_i} \big(\{\varphi \in \mathcal X: \|\varphi - \Phi_\cdot (\theta^{N_i}_0, \omega^{N_i}_0) \|_{\mathcal X} >\eps_0 \} \big) \geq \eps_0 , \quad i\geq 1.
  \end{equation}

First, since $\sup_{i\geq 1} \|(\theta^{N_i}_0, \omega^{N_i}_0) \|_{L^2} \leq R$, up to a subsequence, we can assume that $(\theta^{N_i}_0, \omega^{N_i}_0)$ converges weakly in $(L^2(\T^2))^2$ to some $(\theta_0,\omega_0)$. For any $i\geq 1$, let $(\theta^{N_i}, \omega^{N_i})$ be a weak solution to \eqref{stoch-2D-Boussinesq-Ito} in the sense of Definition \ref{2-def}, with $\sigma = \sigma^{N_i}$ and distributed as $Q^{N_i}$; in particular, it has the initial data $(\theta^{N_i}_0, \omega^{N_i}_0)$. Using again the boundedness of the family $\{(\theta^{N_i}_0, \omega^{N_i}_0) \}_{i\geq 1}$ and by Theorem \ref{thm-existence}, we have
  $$\sup_{i\geq 1} \Big(\|\theta^{N_i} \|_{L^\infty(L^2)} \vee \|\nabla\theta^{N_i} \|_{L^2(L^2)}\Big) \bigvee \sup_{i\geq 1} \|\theta^{N_i} \|_{L^\infty(L^2)} \leq C_0 <\infty ;$$
moreover, $(\theta^{N_i}, \omega^{N_i})$ satisfies stochastic equations of the form \eqref{approx.1} and \eqref{approx.2}. Therefore, we can repeat the arguments in the proof of Proposition \ref{prop-weak-convergence} to show that, up to a further subsequence, $(\theta^{N_i}, \omega^{N_i})$ converges weakly to the unique solution $\Phi_\cdot (\theta_0,\omega_0)$ of the deterministic system \eqref{prop-weak-convergence.0} with initial data $(\theta_0,\omega_0)$. Since the limit is deterministic, we conclude that $(\theta^{N_i}, \omega^{N_i})$ converges also in probability to $\Phi_\cdot (\theta_0,\omega_0)$. This implies that
  \begin{equation}\label{proof-converg.2}
  \lim_{i\to \infty} Q^{N_i} \big(\{\varphi \in \mathcal X: \|\varphi - \Phi_\cdot (\theta_0,\omega_0) \|_{\mathcal X} >\eps_0 \} \big) =0. \end{equation}

Next, for any $i\geq 1$, recall that $\Phi_\cdot (\theta^{N_i}_0, \omega^{N_i}_0)$  is the unique solution to the deterministic system \eqref{prop-weak-convergence.0} with initial data $(\theta^{N_i}_0, \omega^{N_i}_0)$. Using the equations in \eqref{prop-weak-convergence.0} and the weak convergence of  $(\theta^{N_i}_0, \omega^{N_i}_0)$ to $(\theta_0,\omega_0)$, it is easy to see that, up to a subsequence, $\Phi_\cdot (\theta^{N_i}_0, \omega^{N_i}_0)$ converges in the topology of $\mathcal X$ to the limit $\Phi_\cdot(\theta_0,\omega_0)$, which is the solution of the system \eqref{prop-weak-convergence.0} with initial data $(\theta_0,\omega_0)$. Combining this fact with \eqref{proof-converg.2}, we immediately get a contradiction to \eqref{proof-converg.1}. This completes the proof.
\end{proof}

\section{Appendix: Uniqueness of viscous Boussinesq system with $L^2$-initial data}

In this section we prove that the viscous system
  \begin{equation}\label{viscous-2D-Boussinesq}
  \left\{\aligned
  & \partial_t \theta + u\cdot\nabla \theta = \kappa\Delta \theta,\\
  & \partial_t\omega + u\cdot\nabla \omega = \partial_1\theta + \nu \Delta \omega, \\
  & u= K\ast \omega
  \endaligned \right.
  \end{equation}
admits a unique global solution for $L^2$-initial data $(\theta_0, \omega_0)$. Since the precise values of $\kappa$ and $\nu$ are not important in the arguments below, we assume $\kappa=\nu =1$ for simplicity.

We first give some classical a priori estimates. Since the velocity field $u$ is divergence free, the first equation in \eqref{viscous-2D-Boussinesq} yields
  \begin{equation}\label{a-priori-1}
  \|\theta_t\|_{L^2}^2 + 2 \int_0^t \|\nabla\theta_s \|_{L^2}^2\,\d s = \|\theta_0\|_{L^2}^2, \quad t\in [0,T].
  \end{equation}
Using the second equation we have
  \begin{equation}\label{a-priori-2.1}
  \frac{\d}{\d t}\|\omega \|_{L^2}^2= 2\<\omega, \Delta\omega + \partial_1\theta \> \leq -2\|\nabla\omega \|_{L^2}^2 +2 \|\omega \|_{L^2} \|\partial_1\theta \|_{L^2}.
  \end{equation}
In particular, $\frac{\d}{\d t}\|\omega \|_{L^2}^2 \leq 2 \|\omega \|_{L^2} \|\partial_1\theta \|_{L^2}$, and thus $\frac{\d}{\d t}\|\omega \|_{L^2} \leq \|\partial_1\theta \|_{L^2} \leq \|\nabla\theta \|_{L^2}$, which implies
  \begin{equation}\label{a-priori-2.2}
  \aligned
  \|\omega_t \|_{L^2} &\leq \|\omega_0 \|_{L^2} + \int_0^t \|\nabla\theta_s \|_{L^2}\,\d s \\
  &\leq \|\omega_0 \|_{L^2} + \sqrt{T} \|\nabla\theta \|_{L^2(L^2)} \leq (1\vee \sqrt T) (\|\omega_0 \|_{L^2} + \|\theta_0\|_{L^2}),
  \endaligned
  \end{equation}
where $\|\cdot \|_{L^2(L^2)}= \|\cdot \|_{L^2(0,T; L^2(\T^2))}$ and we have used \eqref{a-priori-1} in the last step. Integrating  \eqref{a-priori-2.1} in time yields
  $$\aligned
  \|\omega_t \|_{L^2}^2 + 2\int_0^t \|\nabla\omega_s \|_{L^2}^2 \,\d s &\leq \|\omega_0 \|_{L^2}^2 + 2 \int_0^t \|\omega_s \|_{L^2} \|\partial_1\theta_s \|_{L^2} \,\d s \\
  &\leq \|\omega_0 \|_{L^2}^2 + 2 \|\omega \|_{L^2(L^2)} \|\nabla\theta \|_{L^2(L^2)} .
  \endaligned $$
Combining this estimate with \eqref{a-priori-1} and \eqref{a-priori-2.2}, we arrive at
  \begin{equation}\label{a-priori-2}
  \|\omega_t \|_{L^2}^2 + 2\int_0^t \|\nabla\omega_s \|_{L^2}^2 \,\d s \leq C_T \big(\|\omega_0 \|_{L^2}^2 + \|\theta_0 \|_{L^2}^2 \big).
  \end{equation}
From the estimates \eqref{a-priori-1} and \eqref{a-priori-2}, an application of the Galerkin approximation yields the existence of solutions to the system \eqref{viscous-2D-Boussinesq}.

\begin{theorem}\label{thm-uniqueness}
Given $(\theta_0, \omega_0)\in (L^2(\T^2))^2$, the system \eqref{viscous-2D-Boussinesq} admits a unique solution in
  $$\theta, \omega\in L^\infty \big(0,T; L^2(\T^2) \big) \cap L^2 \big(0,T; H^1(\T^2) \big); $$
moreover, there exists a constant $C_T>0$ such that the solution fulfils the bounds
  \begin{equation}\label{thm-uniqueness.1}
  \|\theta \|_{L^\infty(L^2)} \vee \|\theta \|_{L^2(H^1)} \vee \|\omega \|_{L^\infty(L^2)} \vee \|\omega \|_{L^2(H^1)} \leq C_T \big(\|\omega_0 \|_{L^2} + \|\theta_0 \|_{L^2} \big).
  \end{equation}
\end{theorem}

The proof of the uniqueness assertion is similar to that of 2D Navier-Stokes equations. We introduce the following space
  $$\mathcal H= \big\{ f\in L^2\big(0,T; H^1(\T^2) \big): \partial_t f\in L^2\big(0,T; H^{-1}(\T^2) \big) \big\}, $$
where the time derivative is understood in the distributional sense. It is a classical result (cf. \cite[p. 39, Lemma 2.1.5]{Kuk-Shiri}) that $\mathcal H$ is continuously embedded into $C\big([0,T], L^2(\T^2) \big)$ and compactly embedded in $L^2 \big([0,T], H^m(\T^2) \big)$ for any $m<-1$. Moreover, by \cite[(2.10)]{Kuk-Shiri}, for any $f\in \mathcal H$ and $t\in [0,T]$,
  $$\int_0^t \<f, \partial_s f\>\,\d s = \frac12 \big(\|f_t\|_{L^2}^2 - \|f_0\|_{L^2}^2\big).$$
The above equality implies that $t\mapsto \|f_t\|_{L^2}^2$ is an absolutely continuous function and
  \begin{equation}\label{proof-uniq-0}
  \frac{\d}{\d t} \|f_t\|_{L^2}^2 = 2 \<f, \partial_t f\>.
  \end{equation}
Using the estimates \eqref{thm-uniqueness.1} and the equations in \eqref{viscous-2D-Boussinesq}, it is not difficult to see that $\theta, \omega \in \mathcal H$.

\begin{proof}[Proof of Theorem \ref{thm-uniqueness}:  uniqueness.]
Let $(\theta^i,\omega^i), \, i=1,2$ be two solutions to the system \eqref{viscous-2D-Boussinesq} with the same initial condition $(\theta_0,\omega_0)$; denote by $u^i=K\ast \omega^i,\, i=1,2$. Let $\bar\theta= \theta^1 -\theta^2$, $\bar\omega= \omega^1 -\omega^2$ and $\bar u= u^1-u^2$, then
  $$\aligned
  \partial_t \bar\theta - \Delta \bar\theta &= -\big(\bar u\cdot\nabla \theta^1 + u^2 \cdot\nabla \bar\theta \big), \\
  \partial_t \bar\omega - \Delta \bar\omega &= \partial_1\bar\theta -\big(\bar u\cdot\nabla \omega^1 + u^2 \cdot\nabla \bar\omega \big).
  \endaligned $$
Using \eqref{proof-uniq-0}, we arrive at
  \begin{equation}\label{proof-uniq-1}
  \aligned
  \frac12 \|\bar \theta_t\|_{L^2}^2 + \int_0^t \|\nabla \bar\theta_s \|_{L^2}^2 \,\d s &= - \int_0^t \big[ \< \bar u_{s}\cdot\nabla\theta^1_{s}, \bar\theta_s\> +\< u^2_{s}\cdot\nabla\bar\theta_s, \bar\theta_s \> \big] \,\d s= - \int_0^t \< \bar u_{s}\cdot\nabla\theta^1_{s}, \bar\theta_s\> \,\d s
  \endaligned
  \end{equation}
since $u^2_s $ is divergence free. In the same way, using the second equation yields
  \begin{equation}\label{proof-uniq-2}
  \frac12 \|\bar \omega_t\|_{L^2}^2 + \int_0^t \|\nabla \bar\omega_s \|_{L^2}^2 \,\d s = - \int_0^t  \<\bar u_s\cdot\nabla \omega^1_s, \bar\omega_s\> \,\d s + \int_0^t \<\partial_1 \bar\theta_s, \bar \omega_s\> \,\d s.
  \end{equation}

We first estimate the right-hand side of \eqref{proof-uniq-1}: by H\"older's inequality with exponents $\frac12+ \frac14+\frac14 =1$,
  $$\aligned
  \big| \<\bar u_s \cdot\nabla\theta^1_s, \bar\theta_s \> \big| &\leq \|\nabla\theta^1_s \|_{L^2} \|\bar\theta_s \|_{L^4} \|\bar u_s \|_{L^4}  \leq C  \|\theta^1_s \|_{H^1} \|\bar\theta_s \|_{H^1} \|\bar u_s \|_{H^1} ,
  \endaligned $$
where we have used the Sobolev embedding inequality $(H^1 \subset )\, H^{1/2} \subset L^4$ in the 2D setting. Since $\|\bar\theta_s \|_{H^1} \leq C \|\nabla \bar\theta_s \|_{L^2}$, one has
  $$\big| \<\bar u_s \cdot\nabla\theta^1_s, \bar\theta_s \> \big| \leq C \|\theta^1_s \|_{H^1} \|\nabla \bar\theta_s \|_{L^2} \|\bar u_s \|_{H^1} \leq \frac12 \|\nabla \bar\theta_s \|_{L^2}^2 + C \|\theta^1_s \|_{H^1}^2 \|\bar \omega_s \|_{L^2}^2 .$$
Combining this estimate with \eqref{proof-uniq-1} leads to
  \begin{equation}\label{proof-uniq-3}
  \|\bar \theta_t\|_{L^2}^2 + \int_0^t \|\nabla \bar\theta_s \|_{L^2}^2 \,\d s \leq C \int_0^t \|\theta^1_s \|_{H^1}^2 \|\bar \omega_s \|_{L^2}^2 \,\d s.
  \end{equation}
Similarly, for the right-hand side of \eqref{proof-uniq-2}, one has
  $$\big| \<\bar u_s\cdot\nabla \omega^1_s, \bar\omega_s\> \big| \leq \|\nabla \bar\omega_s \|_{L^2}^2 + C \|\omega^1_s \|_{H^1}^2 \|\bar \omega_s \|_{L^2}^2$$
and
  $$|\<\partial_1 \bar\theta_s, \bar \omega_s\>| \leq \|\partial_1 \bar\theta_s \|_{L^2} \|\bar \omega_s \|_{L^2} \leq \|\nabla \bar\theta_s \|_{L^2} \|\bar \omega_s \|_{L^2} \leq \frac12 \|\nabla \bar\theta_s \|_{L^2}^2 + \frac12 \|\bar \omega_s \|_{L^2}^2. $$
Therefore, we obtain from \eqref{proof-uniq-2} that
  $$\|\bar \omega_t\|_{L^2}^2 \leq C \int_0^t \big(1 +\|\omega^1_s \|_{H^1}^2 \big) \|\bar\omega_s \|_{L^2}^2 \,\d s + \int_0^t \|\nabla \bar\theta_s \|_{L^2}^2 \,\d s. $$
Combining this inequality with \eqref{proof-uniq-3} gives us
  $$\aligned
  \|\bar \theta_t\|_{L^2}^2 + \|\bar \omega_t\|_{L^2}^2 &\leq C \int_0^t \big(1 +\|\omega^1_s \|_{H^1}^2 \big) \|\bar\omega_s \|_{L^2}^2 \,\d s + C \int_0^t \|\theta^1_s \|_{H^1}^2 \|\bar \omega_s \|_{L^2}^2 \,\d s  \\
  &\leq C \int_0^t \big(1 +\|\omega^1_s \|_{H^1}^2 + \|\theta^1_s \|_{H^1}^2 \big) \|\bar\omega_s \|_{L^2}^2  \,\d s.
  \endaligned $$
Since $1 +\|\omega^1_s \|_{H^1}^2 + \|\theta^1_s \|_{H^1}^2$ is integrable in $s\in [0,T]$ and $\|\bar\omega_s \|_{L^2}^2 \leq \|\bar \theta_s\|_{L^2}^2 + \|\bar \omega_s\|_{L^2}^2$, the Gronwall inequality implies that $\|\bar \theta_t\|_{L^2}^2 + \|\bar \omega_t\|_{L^2}^2 =0$ for all $t\in [0,T]$. Therefore, $\theta^1= \theta^2$ and $\omega^1 = \omega^2$. This completes the proof of uniqueness.
\end{proof}

\bigskip

\noindent\textbf{Acknowledgement.} The author is grateful to the financial supports of the National Natural Science Foundation of China (Nos. 11688101, 11931004), and the Youth Innovation Promotion Association, CAS (2017003).

\end{document}